\documentclass[11pt,leqno]{amsart}
\usepackage{amsmath,amssymb,amsxtra,color,calligra,mathrsfs,comment,url}

\usepackage{xcolor} 
\colorlet{mdtRed}{red!50!black}
\definecolor{dblue}{rgb}{0,0,.6}
\usepackage[colorlinks]{hyperref}
\hypersetup{linkcolor=blue,citecolor=dblue,filecolor=dullmagenta,urlcolor=mdtRed}

\usepackage[all]{xy}
\usepackage{tikz,tikz-cd,tkz-graph,enumerate}
\usetikzlibrary{matrix,arrows,decorations.pathmorphing}

\DeclareMathOperator{\GL}{\textnormal{GL}}

\DeclareMathOperator{\Rep}{\textnormal{Rep}}

\DeclareMathOperator{\et}{\textnormal{\'et}}

\newcommand{\mc}[1]{\mathcal{#1}}

\newtheorem{theorem-intro}[]{Theorem}

\numberwithin{equation}{subsection}

\newtheorem{theorem}[equation]{Theorem}

\newtheorem{lemma}[equation]{Lemma}
\newtheorem{proposition}[equation]{Proposition}
\newtheorem{definition}[equation]{Definition}

\theoremstyle{definition}
\newtheorem{remark}[equation]{Remark}

\begin{document}

\title[On Fundamental Group-Schemes]{Fundamental Group Schemes of $n$-fold Symmetric Product of a Smooth Projective Curve} 

\author[A. Paul]{Arjun Paul} 

\address{Department of Mathematics, Indian Institute of Technology Bombay, Powai, Mumbai 400076, Maharashtra, India.} 

\email{arjun.math.tifr@gmail.com} 

\author[R. Sebastian]{Ronnie Sebastian} 

\address{Department of Mathematics, Indian Institute of Technology Bombay, Powai, Mumbai 400076, Maharashtra, India.} 

\email{ronnie@math.iitb.ac.in} 




\begin{abstract}
	Let $k$ be an algebraically closed field of characteristic $p > 0$. 
	Let $X$ be an irreducible smooth projective curve of genus $g$ over $k$. 
	Fix an integer $n \geq 2$, and let $S^n(X)$ be the $n$-fold symmetric product of $X$. 
	In this article we find the $S$-fundamental group scheme and Nori's 
	fundamental group scheme of $S^n(X)$. 
\end{abstract}

\maketitle

\section{Introduction}
For a connected reduced complete scheme $X$ defined over a perfect field $k$ and having a $k$-rational 
point $x$, in \cite{No1, No2}, Nori introduced an affine $k$-group scheme $\pi^N(X, x)$ associated to 
the neutral Tannakian category of essentially finite vector bundles on $X$, known as 
{\it Nori's fundamental group scheme}. This group scheme carries more informations 
than the \'etale fundamental group scheme $\pi^{\et}(X, x)$ in positive characteristic, and is the same as 
$\pi^{\et}(X, x)$ when $k = \mathbb{C}$. For a connected smooth projective curve defined over an 
algebraically closed field $k$, in \cite{BPS}, Biswas, Parameswaran and Subramanian defined 
and studied the {\it $S$-fundamental group scheme} $\pi^S(X, x)$ of $X$. 
This is further generalized and extensively studied for higher dimensional smooth projective varieties over 
algebraically closed fields by Langer in \cite{La, La2}. 
It is an interesting question to find $\pi^{\et}(X, x)$, $\pi^N(X, x)$ and $\pi^S(X, x)$ 
for well-known algebraic varieties.

Let $X$ be a connected smooth projective curve defined over an algebraically closed field $k$ of characteristic $p > 0$. 
Fix an integer $n \geq 2$, and let $S_n$ be the permutation group of $n$ symbols. Then $S_n$ acts on $X^n$ by 
permutation of its factors, and the associated quotient $S^n(X) := X^n/S_n$ is a connected smooth projective 
variety over $k$. For any affine $k$-group scheme $G$ we denote by $G_{\rm ab}$ its abelianization. 
In this article we prove the following results. 

\begin{theorem-intro}[Theorem \ref{main-thm-S-fgs}]
	For any closed point $x \in X(k)$, there is an isomorphism of affine $k$-group schemes 
	$$\widetilde{\psi_*^S} : \pi^S(X, x)_{\rm ab} \longrightarrow \pi^S(S^n(X), nx).$$
\end{theorem-intro}

\begin{theorem-intro}[Theorem \ref{main-thm-N-fgs}]
	For any closed point $x \in X(k)$, there is an isomorphism of affine $k$-group schemes 
	$$\widetilde{\psi_*^N} : \pi^N(X, x)_{\rm ab} \longrightarrow \pi^N(S^n(X), nx).$$ 
\end{theorem-intro}

As a consequence we also obtain the following result, 
which is already contained in \cite{BH}, and proved using a different method.
	For any closed point $x \in X(k)$, there is an isomorphism of affine $k$-group schemes 
	$$\widetilde{\psi_*^{\et}} : \pi^{\et}(X, x)_{\rm ab} \longrightarrow \pi^{\et}(S^n(X), nx).$$ 
Note that when $n>2g-2$, where $g$ is the genus of $X$, these isomorphisms can be easily obtained from results in 
\cite[Section 7]{La2}, since $S^n(X)$ is a projective bundle over ${\rm Alb}(X)$. 
We prove the above results without any restriction on $n$. 
Our initial strategy was to use the same 
method as in \cite{PS19} under the assumption that ${\rm char}(k) > 3$. 
However, we observed that one can avoid using the 
characterization of numerically flat sheaves as strongly semistable reflexive sheaves with vanishing 
Chern classes; proved in \cite{La}. Instead, we first show that $\widetilde{\psi_*^S}$
is faithfully flat and then use \cite[Section 7]{La2} to conclude that 
it is an isomorphism.

\section{Fundamental Group Schemes}
Let $k$ be an algebraically closed field. 
Let $X$ be a reduced proper $k$-scheme, which is connected in the sense that $H^0(X, \mc O_X) \cong k$. 

\subsection{$S$-fundamental group scheme}
Let ${\rm Coh}(X)$ be the category of coherent sheaf of $\mathcal{O}_X$-modules on $X$. This has a full 
subcategory ${\rm Vect}(X)$, whose objects are locally free coherent sheaves (vector bundles) on $X$. 
A vector bundle $E$ on $X$ is said to be {\it nef} if $\mathcal{O}_{\mathbb{P}(E)}(1)$ is a nef line bundle 
on $\mathbb{P}(E)$. 
An object $E$ of ${\rm Coh}(X)$ is said to be {\it numerically flat} if $E$ is locally free and both 
$E$ and its dual $E^{\vee}$ are nef. 
Let $\mathscr{C}^{\rm nf}(X)$ be the full subcategory of ${\rm Coh}(X)$, whose objects are numerically 
flat vector bundles on $X$. 
It is known that, $E \in {\rm Ob}({\rm Coh}(X))$ is an object of $\mathscr{C}^{\rm nf}(X)$ if and only if 
$E$ is locally free and for any smooth projective curve $C$ over $k$ and any morphism $f : C \longrightarrow X$, 
its pullback $f^*E$ on $C$ is slope semistable and of degree $0$ (see \cite[Remark 5.2]{La}). 
Note that $\mathscr{C}^{\rm nf}(X)$ is closed under finite direct sum and tensor products. 
Choosing a closed point $x \in X(k)$, one can define a fiber functor 
$$T_x : \mathscr{C}^{\rm nf}(X) \longrightarrow {\rm Vect}_k$$ 
by sending an object $E$ of $\mathscr{C}^{\rm nf}(X)$ to its fiber $E_x$ at $x$. 
The quadruple $(\mathscr{C}^{\rm nf}(X), \otimes, \mc O_X, T_x)$ is a neutral Tannakian category 
(see \cite[Proposition 5.5]{La}), and the affine $k$-group scheme $\pi^S(X, x)$ Tannaka dual to this is 
known as the {\it $S$-fundamental group scheme} of $X$ with base point $x$.

Let $X$ be a connected smooth projective variety of dimension $d$ over $k$. 
Fix an ample divisor $H$ on $X$. Let ${\rm Vect}_0^s(X)$ be the full subcategory of ${\rm Coh}(X)$, 
whose objects are reflexive coherent sheaves $E$ on $X$, that are strongly $H$-semistable and 
${\rm ch}_1(E)\cdot H^{d-1} = {\rm ch}_2(E)\cdot H^{d-2} = 0$, where ${\rm ch}_i(E)$ is the $i$-th 
Chern character of $E$. 
It is shown in \cite[Proposition 5.1]{La} that the objects of the category ${\rm Vect}_0^s(X)$ are 
in fact locally free coherent sheaves on $X$ and all of their Chern classes vanishes. 
It follows from \cite[Proposition 4.5]{La} that the category ${\rm Vect}_0^s(X)$ does not depend on 
choice of $H$. For $X$ smooth, the categories $\mathscr{C}^{\rm nf}(X)$ and ${\rm Vect}_0^s(X)$ are 
the same (see \cite[Proposition 5.1]{La}, \cite[Theorem 2.2]{La2}). We will not 
use this characterization here, however, this was crucial in \cite{PS19}.

It is clear from the definition of the categories ${\rm Vect}_0^s(X)$ and ${\rm EF}(X)$ that 
$\pi^S(X, x)$ carries more informations than $\pi^N(X, x)$. 
In fact, there are natural faithfully flat homomorphisms of affine $k$-group schemes 
$\pi^S(X, x) \longrightarrow \pi^N(X, x) \longrightarrow \pi^{\et}(X, x)$, (see \cite[Lemma 6.2]{La}).

\subsection{Nori's fundamental group scheme}

\begin{definition}
	A vector bundle $E$ on $X$ is said to be {\it finite} if there are two distinct non-zero polynomials 
	$f$ and $g$ with positive integer coefficients such that $f(E) \cong g(E)$. 
	
	A vector bundle $E$ on $X$ is said to be {\it essentially finite} if there are finitely many finite 
	vector bundles $E_1, \ldots, E_n$ and two numerically flat vector bundles $V_1$ and $V_2$ with 
	$V_2 \subseteq V_1 \subseteq \bigoplus\limits_{i=1}^n E_i$ such that $E \cong V_1/V_2$. 
\end{definition}

Let ${\rm EF}(X)$ be the full subcategory of ${\rm Vect}(X)$ whose objects are essentially finite vector 
bundles on $X$. Then ${\rm EF}(X)$ is an abelian rigid tensor category. 
Let ${\rm Vect}_k$ be the category of $k$-vector spaces. Fixing a closed point $x \in X(k)$, 
we have a fiber functor 
$$T_x : {\rm EF}(X) \longrightarrow {\rm Vect}_k$$
defined by sending a vector bundle $E \in {\rm Ob}({\rm EF}(X))$ to its fiber $E_x$ at $x$. 
This makes the quadruple $({\rm EF}(X), \otimes, \mc O_X, T_x)$ a neutral Tannakian category. 
The affine $k$-group scheme $\pi^N(X, x)$ Tannaka dual to this category is called  
{\it Nori's fundamental group scheme} of $X$ with base point $x$.

\section{Fundamental Group Schemes of $S^n(X)$}
\subsection{Symmetric product of curve}
Let $k$ be an algebraically closed field of characteristic $p > 0$. 
Let $X$ be an irreducible smooth projective curve over $k$. 
Fix an integer $n \geq 2$, and let us denote by $S_n$ the permutation group of $n$ symbols. 
There is a natural action of $S_n$ on the $n$-fold product $X^n$, and the associated 
quotient $X^n/S_n$, denoted by $S^n(X)$, is a smooth projective variety of dimension $n$ over $k$. 
Note that any closed point $q \in S^n(X)$ can be uniquely written as $\sum\limits_{i=1}^r n_ix_i$, 
where $x_1, \ldots, x_r$ are distinct closed points of $X$ and $n_1, \ldots, n_r$ are integers 
with 
\begin{equation*}
n_1 \geq \ldots \geq n_r \geq 1. 
\end{equation*}
We call $\langle n_1, \ldots, n_r \rangle$ the {\it type} of $q$. 
The quotient morphism 
\begin{equation}\label{quotient-map}
	\psi : X^n \longrightarrow S^n(X) 
\end{equation}
is a faithfully flat finite morphism of $k$-schemes. 

\subsection{A group theoretic lemma} 
A proof of the following easy Lemma can be found in \cite{PS19}.
\begin{lemma}\label{group-lemma}
	Let $G$ and $H$ be two group schemes over $k$. For an integer $n \geq 2$, we denote by $G^n$ the 
	group scheme $G \times \cdots \times G$. 
	Then $S_n$ acts on $G^n$ by permuting the factors. 
	Let $f_0$ be the following composite group homomorphism 
	$$f_0 : G^n \stackrel{\alpha^n}{\longrightarrow} (G_{\rm ab})^n \stackrel{m}{\longrightarrow} G_{\rm ab}\,,$$ 
	where $\alpha : G \to G_{\rm ab} := G/[G, G]$ denotes the abelianization homomorphism and $m$ denotes the 
	multiplication homomorphism. 
	Then a homomorphism of $k$-group schemes $f : G^n \longrightarrow H$ is $S_n$-invariant if and only if there 
	is a homomorphism $\tilde f : G_{\rm ab} \longrightarrow H$ of affine $k$-group schemes such that 
	$\tilde f \circ f_0 = f$. In other words, $f$ is $S_n$-invariant iff there if $\tilde f$ which makes 
	the following diagram commutes. 
	\[
	\xymatrix{
		G^{n}\ar[rr]^f\ar[rd]_{f_0} && H\\
		& G_{\rm ab}\ar[ru]_{\tilde f} & 
	}
	\] 
\end{lemma}

\subsection{Construction of homomorphism}
The functor which sends $E\in \mathscr{C}^{\rm nf}(S^n(X))$ to 
$\psi^*E\in \mathscr{C}^{\rm nf}(X^n)$ 
defines a morphism of neutral Tannakian categories (for any closed point $p \in X^n(k)$)
\begin{equation}\label{functor-F}
	\mathscr{F} : (\mathscr{C}^{\rm nf}(S^n(X)), \otimes, \mc O_{S^n(X)}, T_{\psi(p)}) \to 
	(\mathscr{C}^{\rm nf}(X^n), \otimes, \mc O_{X^n}, T_p). 
\end{equation}
Thus, we get a homomorphism
\begin{equation*}
		\psi_*^S : \pi^S(X^n, p) \longrightarrow \pi^S(S^n(X), \psi(p)). 
\end{equation*}

It follows from \cite[Theorem 4.1, p. 842]{La2} that, for any closed point $x \in X(k)$, 
the natural homomorphism of affine $k$-group schemes 
\begin{equation*}
	\pi^S(X^n, (x, \ldots, x)) \stackrel{\simeq}{\longrightarrow} \pi^S(X, x)^n\,. 
\end{equation*}
is an isomorphism. By abuse of notation, denote by $\psi_*^S$ the composite 
of the inverse of this isomorphism and $\psi^S_*$. 
So now 
\begin{equation}\label{hom-of-S-fgs}
\psi^S_*:\pi^S(X,x)^n\to \pi^S(S^n(X),nx),
\end{equation}
where $nx = \psi(x, \ldots, x)$.

The natural $S_n$-action on $X^n$ gives rise to automorphisms 
$\sigma_*$ of the affine $k$-group scheme 
$\pi^S(X^n, (x, \ldots, x)) \cong \pi^S(X, x)^n$, for all $\sigma \in S_n$. 
Now one can check that $\psi_*^S\circ\sigma_* = \psi_*^S$, where $\psi_*^S$ is the homomorphism 
defined in \eqref{hom-of-S-fgs} with $p = (x, \ldots, x)$. 
Consider the natural homomorphism of affine $k$-group schemes 
\begin{equation*}
	\phi : \pi^S(X, x)^n \longrightarrow \pi^S(X, x)_{\rm ab}
\end{equation*}
defined as the following composite homomorphism 
$$\pi^S(X, x)^n \longrightarrow (\pi^S(X, x)_{\rm ab})^n \stackrel{m}{\longrightarrow} \pi^S(X, x)_{\rm ab},$$
where the first homomorphism is given by abelianization at each factor, and the second homomorphism is the 
multiplication. Then it follows from Lemma \ref{group-lemma} that the homomorphism $\psi_*^S$ in 
\eqref{hom-of-S-fgs} factors through a homomorphism 
\begin{equation}\label{pi-S-ab-hom}
	\widetilde{\psi_*^S} : \pi^S(X, x)_{\rm ab} \longrightarrow \pi^S(S^n(X), nx). 
\end{equation}
We record the above discussion in the following proposition.
\begin{proposition}
The map 
$$\psi^S_*:\pi^S(X^n, (x,\ldots,x)) \longrightarrow \pi^S(S^n(X), \psi(x,\ldots,x))$$ 
factors to give a homomorphism 
$\widetilde{\psi_*^S} : \pi^S(X, x)_{\rm ab} \longrightarrow \pi^S(S^n(X), nx)$. 
\end{proposition}

A vector bundle $E$ on $X^n$ is said to be {\it $S_n$-invariant} if $\sigma^*E \cong E$, 
for all $\sigma \in S_n$. 

\begin{proposition}\label{S_n-equivariance-str}
	Let $E$ be a vector bundle in $\mathscr{C}^{\rm nf}(X^n)$ associated to 
	a representation $\rho:\pi^S(X^n,(x,\ldots,x))\cong \pi^S(X,x)^n\to \GL(V)$.
	If $\rho$ factors through $\pi^S(X,x)_{\rm ab}$, as in Lemma \ref{group-lemma},
	then $E$ is $S_n$-invariant.
\end{proposition}

\begin{proof}
From the hypothesis it follows that $\rho\circ \sigma_*=\rho$. The 
proposition follows. 
\end{proof}

\subsection{Faithfully flat}
In this subsection we use \cite[Proposition 2.21]{DMOS} to show that 
the homomorphism $\widetilde{\psi_*^S}$ in \eqref{pi-S-ab-hom} is faithfully flat. 
We begin by recalling this result for the convenience of the reader. 
Let $\theta : G \longrightarrow G'$ 
be a homomorphism of affine group schemes over $k$, and let 
\begin{equation}\label{eqn-hom-f}
	\widetilde{\theta} : \Rep_k(G') \longrightarrow \Rep_k(G) 
\end{equation}
be the functor given by sending $\rho' : G' \to \GL(V)$ to $\rho'\circ \theta : G \to \GL(V)$. 
An object $\rho : G \to \GL(V)$ in $\Rep_k(G)$ is said to be a {\it subquotient} of an object 
$\eta : G \to \GL(W)$ in $\Rep_k(G)$ if there are two $G$-submodules $V_1 \subset V_2$ of $W$ 
such that $V \cong V_2/V_1$ as $G$-modules. 

\begin{proposition}[Proposition 2.21, \cite{DMOS}]\label{DM}
	Let $\theta : G \longrightarrow G'$ be a homomorphism of affine algebraic groups over $k$. Then 
	\begin{enumerate}[(a)]
		\item $\theta$ is faithfully flat if and only if the functor $\widetilde{\theta}$ in \eqref{eqn-hom-f} 
		is fully faithful and given any subobject $W \subset \widetilde{\theta}(V')$, with $V' \in \Rep_k(G')$, 
		there is a subobject $W' \subset V'$ in $\Rep_k(G')$ such that $\widetilde{\theta}(W') \cong W$ in $\Rep_k(G)$. 
		
		\item $f$ is a closed immersion if and only if every object of $\Rep_k(G)$ 
		is isomorphic to a subquotient of an object of the form $\widetilde{\theta}(V')$, 
		for some $V' \in \Rep_k(G')$. 
	\end{enumerate}
\end{proposition}

\begin{proposition}\label{faithfully-flatness}
	The homomorphism 
	$$\widetilde{\psi_*^S} : \pi^S(X, x)_{\rm ab} \longrightarrow \pi^S(S^n(X), nx)$$ 
	defined in \eqref{pi-S-ab-hom} is faithfully flat. 
\end{proposition}

\begin{proof}
	We will apply \cite[Proposition 2.21 (a)]{DMOS}. 
	Let $E_1$ be an object in the category $\mathscr{C}^{\rm nf}(S^n(X))$. 
	Clearly $\psi^*E_1$ has the same rank as that of $E_1$. 
	If $\mc E_2 \subset \mc E_1 := \psi^*E_1$ is a subbundle corresponding to a representation of 
	$\pi^S(X, x)_{\rm ab}$, we need to show that there is a subbundle $E_2 \subset E_1$ such that 
	$\psi^*E_2 = \mc E_2$. 
	We will prove this by induction on rank of $E_1$. If ${\rm rank}(E_1) = 1$, there is 
	nothing to prove. Assume that ${\rm rank}(E_1) \geq 2$. 
	
	The vector bundles $\mc E_i$ corresponds to a representation 
	$$\pi^S(X^n, (x, \ldots, x)) \stackrel{f_0}{\longrightarrow} 
	\pi^S(X, x)_{\rm ab} \stackrel{\rho_i}{\longrightarrow} \GL(V_i), \ \ \forall \ i = 1, 2.$$ 
	It follows from Proposition \ref{S_n-equivariance-str} that 
	$\mc E_2$ is a $S_n$-invariant numerically flat vector bundle on $X^n$. 
	Since $\pi^S(X, x)_{\rm ab}$ is an abelian $k$-group scheme, 
	it follows from \cite[Theorem 9.4, p.~70]{Wa}, that we can find 
	a surjective $\pi^S(X, x)_{\rm ab}$-module homomorphism $V_1 \to L_1$, where $L_1$ is one 
	dimensional and $V_2$ is a $\pi^S(X, x)_{\rm ab}$-submodule of the kernel of this homomorphism. 
	Let $\mc L$ be the line bundle on $X^n$ corresponding to the representation $L_1$. 
	Then it is clear that $\mc L$ is $S_n$-invariant (see Proposition \ref{S_n-equivariance-str}) 
	and there is an $S_n$-equivariant exact sequence of vector bundles 
	$$0 \longrightarrow \mc K \longrightarrow \mc E_1 \longrightarrow \mc{L} \longrightarrow 0$$ 
	on $X^n$ such that $\mc E_2 \subset \mc{K}$. 
	
	Every $S_n$-invariant line bundle on $X^n$ 
	is the pullback of a line bundle from $S^n(X)$ (see \cite[Proposition 3.6]{Fogarty77}, 
	also \cite[Proposition 5.1.1]{PS19}). 
	Therefore, it follows that $L := (\psi_*\mc L)^{S_n}$ is a line bundle on all of $S^n(X)$. 
	We now show that $L$ is numerically flat on $S^n(X)$. 
	Given a morphism $C \longrightarrow S^n(X)$ from a smooth projective curve $C$ into $S^n(X)$, 
	we can find a curve $\widetilde{C}$ and a morphism $\widetilde{C} \longrightarrow C$ making the 
	following diagram commutative. 
	\[
	\xymatrix{
		\widetilde{C} \ar[d] \ar[rr] && X^n \ar[d]^\psi \\ 
		C \ar[rr] && S^n(X) 
	}
	\]
	Since $\mc L$ is numerically flat on $X^n$ and $\mc L \cong \psi^*L$, it follows that $L$ is numerically flat. 
	
	We claim that 
	\begin{equation}\label{eqn-3}
		0\to (\psi_*\mc K)^{S_n} \longrightarrow (\psi_*\mc E_1)^{S_n} \cong E_1 \stackrel{q}{\longrightarrow} 
		(\psi_*\mc L)^{S_n} = L \longrightarrow 0 
	\end{equation}
	is exact. The sequence \eqref{eqn-3} can fail to be exact only on the right. 
	Since both $E_1$ and $L$ are numerically flat and $L$ is a line bundle, $q$ is surjective
	since it is nonzero. 
	This proves the exactness of \eqref{eqn-3}. 
	It follows that $K := (\psi_*\mc K)^{S_n}$ is locally free and numerically flat on $S^n(X)$. 
	It is clear that $\psi^* K = \mc K$ on $X^n$. 
	Since $\mc E_2 \subset \mc K$ the assertion that there is $E_2 \subset E_1$ such that 
	$\mc E_2 = \psi^*E_2$ on $X^n$ follows by induction on rank. 
	
	To complete the proof of the Proposition, we need to show that if 
	$E_1$ and $E_2$ are numerically flat vector bundles on $S^n(X)$ then the natural map 
	$${\rm Hom}_{S^n(X)}(E_1, E_2) \stackrel{}{\longrightarrow} {\rm Hom}_{X^n}(\psi^* E_1, \psi^* E_2)$$  
	is bijective. It is clear that this natural map is injective (faithful). 
	Therefore, it suffices to show that 
	given any numerically flat vector bundle $E$ on $S^n(X)$, any nonzero homomorphism 
	$\phi : \mc O_{X^n} \longrightarrow \psi^*E$ comes from a nonzero homomorphism 
	$\widetilde{\phi} : \mc O_{S^n(X)} \longrightarrow E$. 
	Since the homomorphism $\pi^S(X^n, x) \longrightarrow \pi^S(X, x)_{\rm ab}$ is faithfully flat, 
	and $\psi^*E$ corresponds to a representation of $\pi^S(X, x)_{\rm ab}$, 
	it follows that $\phi$ is a morphism between two representations of $\pi^S(X, x)_{\rm ab}$. 
	This shows that $\phi$ is $S_n$-equivariant on $X^n$. 
	Now from the preceding discussion it follows that $\phi$ arises from a morphism 
	$\mc O_{S^n(X)} \longrightarrow E$. 
\end{proof}

\subsection{Proofs of Theorems}
Let $X$ be a connected smooth projective variety over $k$ and 
$f : \mathbb{P}(E) \longrightarrow X$ be a projective bundle over $X$. 
It is easy to see, using $\pi^S(\mathbb{P}^n,s)=\{1\}$ and \cite[Corollary 12.9, Chapter III]{Ha},
that for a numerically flat sheaf $F$ on $\mathbb{P}(E)$, the sheaf 
$f_*F$ is locally free and the natural map $f^*f_*F\to F$ is an 
isomorphism. From this it easily follows that 
the homomorphism of $S$-fundamental group schemes 
\begin{equation*}\label{f_*^S}
	f_*^S : \pi^S(\mathbb{P}(E), y) \longrightarrow \pi^S(X, f(y)) 
\end{equation*}
is an isomorphism, for all $y \in \mathbb{P}(E)$. 

\begin{theorem}\label{main-thm-S-fgs}
	The homomorphism of affine $k$-group schemes 
	\begin{equation}
	\widetilde{\psi_*^S} : \pi^S(X, x)_{\rm ab} \longrightarrow \pi^S(S^n(X), nx) \nonumber 
	\end{equation}
	is an isomorphism, for all $x \in X(k)$.  
\end{theorem}

\begin{proof}
Let ${\rm Alb}(X)$ be the Albanese variety of $X$. Let $g$ be the genus of the curve $X$. 
Fix a closed point $x \in X(k)$. 
If $n \geq 2g-1$, then the morphism $\eta : S^n(X) \to {\rm Alb}(X)$ given by 
$$\sum_{i=1}^n x_i \mapsto \sum_{i=1}^n x_i - nx,$$ 
makes $S^n(X)$ into a projective bundle over ${\rm Alb}(X)$. 
It follows that the induced homomorphism 
of affine $k$-group schemes is an 
isomorphism,
$$\eta_* : \pi^S(S^n(X), nx) \stackrel{\sim}{\longrightarrow} \pi^S({\rm Alb}(X), 0)\,.$$  
From \cite[Section 7]{La2} it follows that the Albanese morphism 
${\rm alb}_X : X \longrightarrow {\rm Alb}(X)$ given by $t \mapsto t-x$ induces maps 
$${\rm alb}_{X*} : \pi^S(X, x)\stackrel{a_0}{\longrightarrow}\pi^S(X, x)_{\rm ab} 
\stackrel{\widetilde{{\rm alb}_{X*}}}{\longrightarrow} \pi^S({\rm Alb}(X), 0),$$
where $\widetilde{{\rm alb}_{X*}}$ is an isomorphism.
Consider the commutative diagram
\[
 \xymatrix{
	X\ar[r]^a\ar[dr]_{\psi\circ a} & X^n\ar[d]^{\psi} & \\
	& S^n(X)\ar[r]^\eta & {\rm Alb}(X)\,,
 }
\]
where $a(t) = (t, x, x, \ldots, x)$.
At the level of $S$-fundamental group schemes we get the commutative diagram
\[
 \xymatrix{
	\pi^S(X,x)\ar[r]^{a_*}\ar[dr]_{a_0} & \pi^S(X,x)^n\ar[d]_{\psi_0}\ar[r] &\pi^S(X,x)^n_{\rm ab}\ar[dl]_m &\\
	&\pi^S(X,x)_{\rm ab}\ar[r]^{\widetilde{\psi^S_*}}& \pi^S(S^n(X), nx) \ar[r]^{\eta_*} & \pi^S({\rm Alb}(X), 0)\,,
 }
\]
We have $\eta\circ \psi\circ a={\rm alb}_X$. It is easy to check that 
$(\psi\circ a)_*=\widetilde{\psi^S_*}\circ a_0$. Since $a_0$ is faithfully flat
and 
$$\widetilde{\psi^S_*}\circ a_0=\eta_*^{-1}\circ \widetilde{{\rm alb}_{X*}}\circ a_0\,,$$
it follows that $\widetilde{\psi^S_*}=\eta_*^{-1}\circ \widetilde{{\rm alb}_{X*}}$
and so the theorem is true is $n\geq 2g-1$.

Assume that $n < 2g-1$.
Consider the maps 
$$X \stackrel{a}{\longrightarrow} X^n \stackrel{\psi}{\longrightarrow} S^n(X) 
\stackrel{c}{\longrightarrow}S^{2g-1}(X) \stackrel{\eta}{\longrightarrow} {\rm Alb}(X),$$ 
where $a(t) = (t, x, x, \ldots, x)$ and $c(\sum_{i=1}^n x_i) = \sum_{i=1}^n x_i + (2g-1-n)x$. 
It is clear that the composite morphism is ${\rm alb}_X$. 
As above, one easily checks that 
$$\eta_*\circ c_*\circ \widetilde{\psi^S_*}=\widetilde{{\rm alb}_{X*}}\,.$$
Thus, we get homomorphisms of affine $k$-group schemes 
$$\pi^S(X, x)_{\rm ab} \stackrel{\widetilde{\psi^S_*}}{\longrightarrow} \pi^S(S^n(X), nx) 
		\stackrel{\eta_*\circ c_*}{\longrightarrow} \pi^S({\rm Alb}(X), 0),$$ 
such that the composite homomorphism is an isomorphism. This forces that the first 
map is a closed immersion. Since we know from Proposition \ref{faithfully-flatness} 
that the first map is faithfully flat, the theorem follows.
\end{proof}

\begin{remark}
That $\widetilde{\psi_*^S}$ is a closed immersion could have been 
proved using the same method in \cite{PS19} under the assumption that ${\rm char}(k)>3$.
\end{remark}

Let $X$ be a reduced proper $k$-scheme with $H^0(X, \mc O_X) = k$. 
Let $E$ be an essentially finite vector bundle on $X$. 
Then there is a finite $k$-group scheme $G$, a principal $G$-bundle $p : P \to X$ 
and a finite dimensional $k$-linear representation $\rho : G \to \GL(W)$ such that 
$E \cong P \times^{\rho} W$, the vector bundle associated to the representation $\rho$. 
It follows from the proof of \cite[Proposition 3.8]{No1} that there 
is a finite vector bundle $\mc V$ on $X$ such that $E$ is a subbundle of $\mc V$.

As before, let $X$ be a connected smooth projective curve over $k$ and $S^n(X)$ 
the $n$-fold symmetric product of $X$. 
It is clear that $\psi^*$
takes a finite vector bundle to a finite vector bundle. 
Thus, $\psi^*E \subset \psi^*\mc V$, which shows that $\mathscr F$ 
takes essentially finite vector bundles to essentially finite vector bundles.  
Note that there is a commutative diagram of homomorphisms of affine $k$-group schemes 
$$
\xymatrix{
	\pi^S(X,x)_{\rm ab}\ar[r]^-{\simeq} \ar[d] & \pi^S(S^n(X),nx) \ar[d]\\
	\pi^N(X,x)_{\rm ab}\ar[r]^-{\widetilde{\psi_*^N}} & \pi^N(S^n(X), nx) 
}
$$
where the vertical arrows are faithfully flat by \cite[Lemma 6.2]{La}. 
It follows that $\widetilde{\psi_*^N}$ is faithfully flat. 

Now let $\mc E$ be an essentially finite $S_n$-invariant vector bundle on $X^n$. 
It is easy to find a finite and $S_n$-invariant bundle $\mc V$ on $X^n$ and an 
$S_n$-equivariant inclusion $\mc E \subset \mc V$. 
Define $E = (\psi_*\mc E)^{S_n}$ and $V := (\psi_*\mc V)^{S_n}$. 
It is clear that $V$ is a finite vector bundle 
and $E \subset V$. So $E$ is essentially finite and $\mathscr F(E) = \mc E$. 
This shows that $\widetilde{f}^N$ is a closed immersion. 
Thus, we have the following. 

\begin{theorem}\label{main-thm-N-fgs}
	There is a natural isomorphism of affine $k$-group schemes 
	$$\widetilde{\psi_*^N} : \pi^N(X, x)_{\rm ab} \longrightarrow \pi^N(S^n(X), nx).$$
\end{theorem}

\subsection{\'Etale Fundamental Group Scheme of $S^n(X)$}\label{subsection-etale} 
In this subsection we sketch how to deduce from Theorem \ref{main-thm-N-fgs} 
the same assertion for \'etale fundamental group schemes. 
This result is a special case of \cite[Theorem 1.2]{BH}. 
Note that there is a commutative diagram
\[
\xymatrix{
	\pi^N(X, x) \ar@{->>}[r]\ar@{->>}[d] & \pi^N(X, x)_{\rm ab} \ar[r]^-{\sim} \ar@{->>}[d] & 
	\pi^N(S^n(X), nx) \ar@{->>}[d]^{d}\\ 
	\pi^{\et}(X, x) \ar@{->>}[r] & \pi^{\et}(X, x)_{\rm ab} \ar[r] & \pi^{\et}(S^n(X), nx)\,.
}
\]
From this it follows that $\pi^{\et}(X,x)_{\rm ab} \longrightarrow \pi^{\et}(S^n(X), nx)$ 
is faithfully flat. Consider a homomorphism $\pi^{\et}(X, x)_{\rm ab} \to \GL(V)$. It follows 
using \cite[Proposition 3.10]{No1} that this homomorphism factors through a finite and reduced 
$k$-group scheme $G$. Now consider the diagram 
\[
\xymatrix{
	\pi^N(X, x)_{\rm ab} \ar[r]^-{\sim} \ar@{->>}[d] & \pi^N(S^n(X), nx) \ar@{->>}[r]^{d} 
	\ar@{-->}[d] & \pi^{\et}(S^n(X), nx) \ar@{-->}[dl]\\
	\pi^{\et}(X, x)_{\rm ab} \ar[r] & G \ar[r] & \GL(V)\,.
}
\]
The right vertical arrow is the unique map which makes the square commute. It factors through 
$d$ since $G$ is finite and reduced. Now it follows from \cite[Proposition 2.21 (b)]{DMOS} 
that $\pi^{\et}(X,x)_{\rm ab} \longrightarrow \pi^{\et}(S^n(X), nx)$ is a closed immersion. 
This proves the following. 
\begin{theorem}\label{main-thm-etale-fgs}
	For any closed point $x \in X(k)$, there is an isomorphism of affine $k$-group schemes 
	$$\widetilde{\psi_*^{\et}} : \pi^{\et}(X, x)_{\rm ab} \longrightarrow \pi^{\et}(S^n(X), nx).$$ 
\end{theorem}


\end{document}